\documentclass{amsart}
\usepackage[all]{xy}
\usepackage{amsmath,amssymb,dsfont,mathtools,enumitem,nicefrac}
\usepackage[pdftitle={Localisation and colocalisation of KK-theory at sets of primes},
  pdfauthor={Hvedri Inassaridze, Tamaz Kandelaki, Ralf Meyer},
  pdfsubject={Mathematics}
]{hyperref}
\usepackage[lite]{amsrefs}
\usepackage{microtype}

\newcommand*{\MRref}[2]{ \href{http://www.ams.org/mathscinet-getitem?mr=#1}{MR \textbf{#1}}}
\newcommand*{\arxiv}[1]{\href{http://www.arxiv.org/abs/#1}{arXiv: #1}}

\newtheorem{thm}{Theorem}[section]
\newtheorem{cor}[thm]{Corollary}
\newtheorem{lem}[thm]{Lemma}
\newtheorem{prop}[thm]{Proposition}
\theoremstyle{definition}
\newtheorem{defn}[thm]{Definition}
\theoremstyle{remark}
\newtheorem{rem}[thm]{Remark}
\newtheorem{example}[thm]{Example}

\numberwithin{equation}{section}

\DeclareMathOperator{\Rep}{Rep}
\DeclareMathOperator{\Hom}{Hom}
\DeclareMathOperator{\Tor}{Tor}
\DeclareMathOperator{\coker}{coker}

\newcommand*{\N}{\mathbb N}
\newcommand*{\Z}{\mathbb Z}
\newcommand*{\Q}{\mathbb Q}
\newcommand*{\R}{\mathbb R}
\newcommand*{\C}{\mathbb C}
\newcommand*{\Sphere}{\mathbb S}

\newcommand*{\Zen}{R}
\newcommand*{\ZenS}{S^{-1}R}
\newcommand*{\Tri}{\mathcal T}
\newcommand*{\TriS}{S^{-1}\Tri}
\newcommand*{\FinS}{\mathcal N_S}
\newcommand*{\KKcat}{\mathcal{KK}}
\newcommand*{\Cat}{\mathcal C}
\newcommand*{\Cats}{\mathcal C_S}
\newcommand*{\prid}{\mathfrak p}
\newcommand*{\Unit}{\mathds 1}
\newcommand*{\Rightco}{\mathbb R^\bot}

\newcommand*{\inOb}{\mathrel{\in\in}\nobreak}

\newcommand*{\Ab}{\mathfrak{Ab}}

\newcommand*{\Cst}{\textup{C}^*}
\newcommand*{\K}{\textup K}
\newcommand*{\KO}{\textup{KO}}
\newcommand*{\Cont}{\textup C}
\newcommand*{\topo}{\textup{top}}
\newcommand*{\red}{\textup{r}}
\newcommand*{\KK}{\textup{KK}}
\newcommand*{\KKO}{\textup{KKO}}
\newcommand*{\id}{\textup{id}}

\newcommand*{\nb}{\nobreakdash}
\newcommand*{\defeq}{\mathrel{\vcentcolon=}}
\newcommand*{\mono}{\rightarrowtail}
\newcommand*{\epi}{\twoheadrightarrow}
\newcommand*{\blank}{\textup{\textvisiblespace}}

\begin{document}
\title[Localisation and colocalisation]{Localisation and colocalisation of KK-theory\\ at sets of primes}
\author{Hvedri Inassaridze}
\author{Tamaz Kandelaki}
\address{H. Inassaridze, T.Kandelaki: A.~Razmadze Mathematical Institute, M.~Alexidze Street~1, 380093 Tbilisi, Georgia}
\email{inassari@gmail.com}
\email{tam.kandel@gmail.com}
\author{Ralf Meyer}
\address{Ralf Meyer: Mathematisches Institut and\\Courant Centre ``Higher order structures'', Georg-August Universit\"at G\"ottingen, Bunsenstra{\ss}e 3--5, 37073 G\"ottingen, Germany}
\email{rameyer@uni-math.gwdg.de}

\thanks{This research was supported by the Volkswagen Foundation (Georgian--German non-commutative partnership).  The third author was supported by the German Research Foundation (Deutsche Forschungsgemeinschaft (DFG)) through the Institutional Strategy of the University of G\"ottingen.}

\begin{abstract}
  Given a set of prime numbers~\(S\), we localise equivariant bivariant Kasparov theory at~\(S\) and compare this localisation with Kasparov theory by an exact sequence.  More precisely, we define the localisation at~\(S\) to be \(\KK^G(A,B) \otimes \Z[S^{-1}]\).  We study the properties of the resulting variants of Kasparov theory.
\end{abstract}

\maketitle

\section{Introduction}
\label{sec:intro}

Localisation at a prime number is a standard technique for computations in stable homotopy theory.  We consider a more general situation here.  Let~\(\Zen\) be a commutative ring and let~\(\Tri\) be an \(\Zen\)\nb-linear triangulated category, that is, the morphism spaces in~\(\Tri\) are enriched to \(\Zen\)\nb-modules and the composition in~\(\Tri\) is \(\Zen\)\nb-bilinear.  Let \(S\subseteq\Zen_\Tri\setminus\{0\}\) be a multiplicatively closed subset.  Let \(\ZenS\) be the localisation of the ring~\(\Zen\) at~\(S\).  We define the localisation of~\(\Tri\) at~\(S\) to be the category~\(\Tri[S^{-1}]\) with morphism spaces
\[
\Tri[S^{-1}](A,B)
\defeq \Tri(A,B) \otimes_{\Zen} \ZenS.
\]
This yields again a triangulated category because \(\ZenS\) is a flat module over~\(\Zen\).  Equivalently, we may describe~\(\Tri[S^{-1}]\) as a localisation of~\(\Tri\) at a certain thick subcategory.  Very similar results are due to Paul Balmer~\cite{Balmer:Spectra_spectra} (see \cite{Balmer:Spectra_spectra}*{Theorem 3.6} and \cite{dellAmbroglio:Tensor_triangular}*{Theorem 2.32}).

For instance, let \(\Zen\defeq \Rep(G)\) be the representation ring of a compact Lie group~\(G\) and let~\(\Tri\) be the equivariant Kasparov category~\(\KKcat^G\).  The ring~\(\Rep(G)\) is always a commutative Noetherean ring, its prime ideals are described in~\cite{Segal:Representation_ring}.  Let~\(\hat{S}\) be a set of prime ideals of~\(\Rep(G)\) and let
\[
S\defeq \Rep(G) \bigm\backslash \bigcup_{\prid\notin \hat{S}} \prid
\]
be the associated multiplicatively closed subset of elements that are invertible outside~\(S\).  The resulting localisation of~\(\KK^G\) at~\(S\) is used, for instance, in~\cite{dellAmbroglio:Tensor_triangular}.

A particularly simple example is rational \(\KK^G\)-theory.  Here we take \(\Zen=\Z\) and \(S=\Z\setminus\{0\}\), so that \(\Z[S^{-1}]=\Q\).  The localisation of \(\KK^G\) at~\(S\) has morphism spaces
\[
\KK^G(A,B;\Q) \defeq \KK^G(A,B) \otimes \Q.
\]
This differs from the definition of Kasparov theory with coefficients in \cite{Blackadar:K-theory}*{Exercise 23.15.6}.  Our definition has several advantages.  Most importantly, the category defined by the groups \(\KK^G(A,B;\Q)\) is again triangulated.

For trivial~\(G\) and \(A=\C\), we get \(\K\)\nb-theory with coefficients in~\(S^{-1}\Z\),
\[
\K_*(A;S^{-1}\Z) \defeq \K_*(A)\otimes S^{-1}\Z;
\]
here our definition is equivalent to the one in \cite{Blackadar:K-theory}*{Exercise 23.15.6}.  \(\K\)\nb-theory with coefficients in~\(\Q\) is particularly important because the Chern character identifies \(\K^*(X)\otimes\Q\) for a compact space~\(X\) with the rational cohomology of~\(X\).

We return to the general case of a localisation of an \(\Zen\)\nb-linear triangulated category~\(\Tri\) at a multiplicatively closed subset~\(S\) of~\(\Zen\).  The embedding \(\Zen\to\ZenS\) induces a canonical map from~\(\Tri\) to~\(\Tri[S^{-1}]\).  Since~\(\Tri[S^{-1}]\) is a localisation of~\(\Tri\), the construction in~\cite{Inassaridze-Kandelaki-Meyer:Localisation_colocalisation} provides a natural long exact sequence
\begin{multline}
  \label{excmain}
  \dotsb \to \Tri_0(A,B;\ZenS/\Zen)
  \to \Tri_0(A,B)
  \to \Tri_0(A,B;\ZenS)
  \\ \to \Tri_{-1}(A,B;\ZenS/\Zen)
  \to \dotsb.
\end{multline}
The groups \(\Tri(A,B;\ZenS/\Zen)\) behave like the morphism spaces in a triangulated category, except that they lack unit morphisms.

For instance, assume \(\Tri=\KKcat\), \(\Zen=\Z\), \(S=\Z\setminus\{0\}\), and \(A=\C\).  Then the torsion theory \(\KK_*(\C,B;S^{-1}\Z/\Z)\) agrees with the \(\K\)\nb-theory of~\(B\) with \(S^{-1}\Z/\Z\)-coefficients, and the long exact sequence~\eqref{excmain} already appears in \cite{Cuntz-Meyer-Rosenberg}*{Section 8.1}.

The definition in~\cite{Inassaridze-Kandelaki-Meyer:Localisation_colocalisation} is not useful to actually compute \(\Tri(A,B;\ZenS/\Zen)\).  To address this problem, recall that
\[
\ZenS/\Zen
\cong \varinjlim x^{-1}\Zen/\Zen
= \varinjlim \Zen/(x),
\]
where \((x)=x\cdot\Zen \cong\Zen\) is the principal ideal generated by~\(x\).  Hence we expect that \(\Tri(A,B;\ZenS/\Zen)\) is a colimit of theories \(\Tri(A,B;s)\) ``with finite coefficients.''  We make this precise below.  Finite coefficient theories on \(\Cst\)\nb-algebras are already considered in~\cite{Schochet:Top4}.

The examples we consider in this article only involve the simple special case \(\Zen=\Z\).  We work in the more general situation described above because the definitions and proofs are all literally the same as in the special case \(\Zen=\Z\).

The exact sequence~\eqref{excmain} allows us to split a problem concerning~\(\Tri\) into two problems concerning \(\Tri[S^{-1}]\) and \(\Tri(A,B;\ZenS/\Zen)\).  The first of these two may be considerably simpler, for suitable choice of~\(S\); the latter involves only torsion \(\Zen\)\nb-modules with the possible torsion controlled by~\(S\).

As an illustration, we consider the Baum--Connes assembly map with coefficients.

Another example is the comparison between real and complex
Kasparov theory.  Let \(A\) and~\(B\) be real \(\Cst\)\nb-algebras
and let \(A_\C\defeq A\otimes_\R\C\), \(B_\C\defeq
B\otimes_\R\C\), then the results in~\cite{Schick:Real_complex}
yield a natural isomorphism
\[
\KK_n^G(A_\C,B_\C;\Z[1/2]) \cong
\KKO_n^G(A,B;\Z[1/2]) \oplus \KKO_{n-2}^G(A,B;\Z[1/2]).
\]

We plan in a forthcoming article to study the situation where~\(G\) is a finite group and~\(S\) is the set of prime numbers dividing the order of~\(G\).  It seems plausible that there should be a tractable Universal Coefficient Theorem computing \(\KK^G_*(A,B;S^{-1}\Z)\) for this choice of~\(S\) for a large class of groups.  Then the long exact sequence~\eqref{excmain} reduces the computation of \(\KK^G_*(A,B)\) to that of the torsion theory \(\KK^G_*(A,B;S^{-1}\Z/\Z)\).

Finally, we consider the exponent of \(\KK^G_*(A,B;\Z/q)\) for \(q\in\N_{\ge2}\).  For complex \(\Cst\)\nb-algebras, it is easy to see that \(q\cdot \KK^G_*(A,B;\Z/q)=\{0\}\) for all~\(q\).  In the real case, we show \(q\cdot \KKO^G_*(A,B;\Z/q)=\{0\}\) for odd~\(q\) and \(2q\cdot \KKO^G_*(A,B;\Z/q)=\{0\}\) for even~\(q\).  We remark without proof that the latter easy result is not optimal: if \(4\mid q\), then \(q\cdot \KKO^G_*(A,B;\Z/q)=\{0\}\) as well, by an analogue of a Theorem by Browder, Karoubi and Lambre for algebraic \(\K\)\nb-theory (see \cites{Browder:Algebraic_K_coefficients, Karoubi-Lambre:Classes_caracteristiques}).

\section{Central localisation and colocalisation}
\label{sec:central_localisation}

Central localisation in the setting of tensor triangulated categories is already studied by Paul Balmer~\cite{Balmer:Spectra_spectra} and Ivo dell'Ambroglio~\cite{dellAmbroglio:Tensor_triangular}.  Here we define it using a tensor product, and we check that two natural definitions are equivalent.  One of them shows that the localisation is again a triangulated category.

Following~\cite{Inassaridze-Kandelaki-Meyer:Localisation_colocalisation}, we then introduce central \emph{co}localisations and describe them more explicitly.  We also consider homology with finite coefficients, assuming the naturality of certain cones.

Let~\(\Zen\) be a commutative unital ring and let~\(S\) be a multiplicatively closed subset of~\(\Zen\).  Let~\(\ZenS\) denote the localisation of~\(\Zen\) at~\(S\) (see~\cite{Atiyah-Macdonald:Commutative}).  This is a unital ring equipped with a natural unital ring homomorphism \(i_S\colon \Zen\to\ZenS\).  The following constructions only depend on this ring extension and not on the choice of~\(S\).

Let~\(\Tri\) be an \(\Zen\)\nb-linear triangulated category, that is, each morphism space in~\(\Tri\) is an \(\Zen\)\nb-module and composition of morphisms is \(\Zen\)\nb-linear.

\subsection{The localisation}
\label{sec:localisation}

\begin{defn}
  \label{def:localisation}
  The \emph{localisation} of~\(\Tri\) at~\(S\) is the \(\ZenS\)\nb-linear additive category~\(\TriS\) with morphism spaces
  \[
  \TriS(A,B) \defeq \Tri(A,B) \otimes_\Zen \ZenS
  \]
  and the obvious composition.  The natural map \(i_S\colon \Zen\to\ZenS\) induces an \(\Zen\)\nb-linear functor \(\Tri\to\TriS\).
\end{defn}

\begin{example}
  \label{exa:integers}
  Any triangulated category is \(\Z\)\nb-linear.  Let~\(S_0\) be a set of prime numbers and let~\(S\) be the set of all natural numbers whose prime factor decomposition contains only primes in~\(S_0\).  This yields a localisation at the set~\(S_0\) of prime numbers.  In particular, we may localise at a single prime number.

  For instance, if~\(S_0\) is the set of all primes, then \(S^{-1}\Z=\Q\) and~\(\TriS\) is a rational version of~\(\Tri\).
\end{example}

\begin{example}
  \label{exa:centre}
  Recall that the centre~\(Z_\Cat\) of an additive category~\(\Cat\) is the set of natural transformations \(\id_\Cat\Rightarrow\id_\Cat\).  That is, an element of the centre is a family of maps \(\phi_x\colon x\to x\) for all objects~\(x\) that is central in the sense that \(f\circ \phi_x=\phi_y\circ f\) for all morphisms \(f\colon x\to y\).  The centre is a commutative unital ring in a natural way, and~\(\Cat\) is \(Z_\Cat\)\nb-linear.  An \(\Zen\)\nb-linear structure on~\(\Cat\) is equivalent to a unital ring homomorphism from~\(\Zen\) to~\(Z_\Cat\).
\end{example}

\begin{example}
  \label{exa:KKG_RepG}
  Let~\(G\) be a compact metrisable group and let \(\Zen = \Rep(G)\) be its representation ring.  Let \(\KKcat^G\) denote the \(G\)\nb-equivariant Kasparov category.  Recall that an object of \(\KKcat^G\) is a separable \(\Cst\)\nb-algebras with a strongly continuous action of~\(G\); the morphism spaces in \(\KKcat^G\) are the \(G\)\nb-equivariant bivariant \(\K\)\nb-groups \(\KK^G(A,B)\) defined by Gennadi Kasparov~\cite{Kasparov:Operator_K}.  The ring \(\Rep(G)\) is isomorphic to \(\KK^G(\C,\C)\) and acts on \(\KK^G(A,B)\) by exterior product.  Thus~\(\KKcat^G\) is \(\Rep(G)\)-linear.

  It is conceivable that the centre of~\(\KKcat^G\) is isomorphic to \(\Rep(G)\), but this seems difficult to prove, even for the trivial group~\(G\), because we know very little about the algebraic properties of \(\KKcat\) beyond the bootstrap category.
\end{example}

\begin{example}
  \label{exa:tensor_triangulated}
  Let~\(\Tri\) be a tensor triangulated category with tensor unit~\(\Unit\).  Then \(\Zen\defeq \Tri(\Unit,\Unit)\) is a commutative ring that acts on \(\Tri(A,B)\) for all \(A,B\) by exterior product, so that~\(\Tri\) becomes \(\Zen\)\nb-linear.  This situation is considered in \cites{Balmer:Spectra_spectra,dellAmbroglio:Tensor_triangular} and contains Example~\ref{exa:KKG_RepG} as a special case.  The following results generalise \cite{Balmer:Spectra_spectra}*{Theorem 3.6} and \cite{dellAmbroglio:Tensor_triangular}*{Theorem 2.32} for this special case.
\end{example}

Our next goal is to realise~\(\TriS\) as a localisation of~\(\Tri\).  This will also show that~\(\TriS\) inherits from~\(\Tri\) a triangulated category structure.  As a preparation, we describe~\(\ZenS\) as a filtered colimit.

\begin{defn}
  \label{def:Cats}
  Let~\(\Cats\) be the category whose object set is~\(S\) and whose morphism space from~\(s\) to~\(t\) is the set of all \(u\in S\) with \(su=t\).  The composition of morphisms in~\(\Cats\) is the multiplication in~\(S\).
\end{defn}

\begin{lem}
  \label{lem:localisation_as_colimit}
  The category~\(\Cats\) is filtered.

  Map \(s\in S\) to the rank-1 free \(\Zen\)\nb-module~\(\Zen\) and the morphism \(u\colon s\to t\) in~\(\Cats\) to the map \(\Zen\to\Zen\), \(r\mapsto u\cdot r\).  This defines an inductive system of \(\Zen\)\nb-modules with colimit~\(\ZenS\).  Let~\(M\) be an \(\Zen\)\nb-module.  If we map \(s\in S\) to~\(M\) and the morphism \(u\colon s\to t\) to the map \(M\to M\), \(f\mapsto u\cdot f\), then we get an inductive system with colimit \(S^{-1}M = \ZenS\otimes_\Zen M\).
\end{lem}

\begin{proof}
  If \(s\) and~\(t\) are objects of~\(\Cats\), then the object~\(st\) dominates both because of the morphisms \(t\colon s\to st\) and \(s\colon t\to st\).  Two morphisms \(u,u'\colon s\rightrightarrows t\) are equalised by \(s\colon t\to st\) because \(su=t=su'\) if \(u,u'\) both map~\(s\) to~\(t\).  Thus~\(\Cats\) is directed.

  It is clear that our prescriptions above yield inductive systems indexed by~\(\Cats\).  Elements of the colimit of this system are equivalence classes of pairs \((a,s)\) with \(a\in\Zen\) (or \(a\in M\)), \(s\in S\), and \((a,s)\) and \((b,t)\) represent the same element in the colimit if there exist \(s',t'\in S\) with \(ss'=tt'\) and \(as'=bt'\).  This equivalence relation on \(M\times S\) for an \(\Zen\)\nb-module~\(M\) describes the localisation~\(S^{-1}M\) (see~\cite{Atiyah-Macdonald:Commutative}).
\end{proof}

\begin{defn}
  \label{def:S-equivalence}
  A morphism~\(f\) in~\(\Tri(A,B)\) is called an \emph{\(S\)\nb-equivalence} if there are \(g,h\in\Tri(B,A)\) and \(s,t\in S\) such that \(g\circ f=s\cdot\id_A\) and \(f\circ h=t\cdot\id_B\).
\end{defn}

\begin{lem}
  \label{lem:S-equivalence}
  A morphism~\(f\) in~\(\Tri(A,B)\) is an \emph{\(S\)\nb-equivalence} if and only if there are \(g\in\Tri(B,A)\) and \(s\in S\) with \(g\circ f=s\cdot\id_A\) and \(f\circ g=s\cdot\id_B\).
\end{lem}

\begin{proof}
  If \(g,h,s,t\) are as in Definition~\ref{def:S-equivalence}, then
  \[
  tg = g\circ f\circ h = sh,\qquad
  (tg)\circ f = st\cdot \id_A,\qquad
  f\circ (sh) = st\cdot \id_B.
  \]
  Thus \(tg=sh\in\Tri(B,A)\) and \(st\in S\) will do.
\end{proof}

\begin{defn}
  \label{def:S-finite}
  An object~\(A\) of~\(\Tri\) is called \emph{\(S\)\nb-finite} if \(s\cdot \id_A=0\) for some \(s\in S\).
\end{defn}

\begin{prop}
  \label{pro:S-finite-equivalence}
  Let \(A\xrightarrow{f} B \to C\to A[1]\) be an exact triangle in~\(\Tri\).  Then the following are equivalent:
  \begin{enumerate}[label=\textup{(\arabic{*})}]
  \item the morphism~\(f\) is an \(S\)\nb-equivalence;
  \item the morphism~\(f\) becomes invertible in~\(\TriS\);
  \item the object~\(C\) is \(S\)\nb-finite;
  \item the object~\(C\) becomes a zero object in~\(\TriS\).
  \end{enumerate}
\end{prop}

\begin{proof}
  (1)\(\iff\)(2): Since the morphism spaces in~\(\TriS\) are \(\ZenS\)-modules and elements of~\(S\) become invertible in~\(\ZenS\), (1) implies that~\(f\) becomes invertible in~\(\TriS\).  Conversely, suppose that~\(f\) becomes invertible in~\(\TriS\).  Lemma~\ref{lem:localisation_as_colimit} shows that its inverse is of the form \(s^{-1}g\) for some \(g\in\Tri(B,A)\), \(s\in S\).  The relation \(f\cdot (s^{-1}g) = \id\) means that \(f \cdot (u g) = us\) for some \(u\in S\).  Similarly, we get \((vg)\cdot f = vs\) for some \(v\in S\).  Hence~\(f\) is an \(S\)\nb-equivalence, so that~(2) implies~(1).

  (3)\(\iff\)(4): The object~\(C\) becomes zero in~\(\TriS\) if and only if the zero map on~\(C\) becomes invertible in~\(\TriS\).  Since there is only one map on the zero object, the inverse map in~\(\TriS\) must be the image of~\(\id_C\).  This means that \(s\cdot \id_C = s\cdot 0\) for some \(s\in S\).  Thus (3)\(\iff\)(4).

  (2)\(\iff\)(4): Since~\(\ZenS\) is a flat \(\Zen\)\nb-module, the functor \(\TriS(D,\blank)\) is homological for any object~\(D\) of~\(\Tri\).  By the Yoneda Lemma, \(f\) becomes invertible in~\(\TriS\) if and only if \(\TriS(D,f)\) is invertible for all~\(D\).  By the long exact sequence for the exact triangle \(A\to B\to C\to A[1]\) and the homological functor~\(\TriS(D,\blank)\), this is equivalent to \(\TriS(D,C)=0\) for all~\(D\).  And this is equivalent to \(C\cong0\) in~\(\TriS\) by another application of the Yoneda Lemma.
\end{proof}

\begin{cor}
  \label{cor:Fins_thick}
  The class~\(\FinS\) of \(S\)\nb-finite objects is a thick subcategory of~\(\Tri\).
\end{cor}

\begin{proof}
  It is clear that~\(\FinS\) is closed under isomorphism and suspension.  It is closed under direct summands because the functor \(\Tri\to\TriS\) is additive.  If \(B\) and~\(C\) in an exact triangle \(A\to B\to C\to A[1]\) are \(S\)\nb-finite, then \(B\cong0\) in~\(\TriS\) and the map \(A\to B\) becomes an isomorphism in~\(\TriS\) by Proposition~\ref{pro:S-finite-equivalence}.  Hence \(A\cong0\) in~\(\TriS\), so that~\(A\) is \(S\)\nb-finite by Proposition~\ref{pro:S-finite-equivalence}.
\end{proof}

\begin{thm}
  \label{the:TriS_localisation_triangulated}
  The category~\(\TriS\) together with the functor \(\Tri\to\TriS\) is the localisation of~\(\Tri\) at the thick subcategory~\(\FinS\).
\end{thm}

\begin{proof}
  It follows from the results above and the standard description of the localisation that the localisation of~\(\Tri\) at~\(\FinS\) is the category theoretic localisation at the class of \(S\)\nb-equivalences (see~\cite{Gabriel-Zisman:Fractions}).

  Let~\(W\) be the set of all morphisms in~\(\Tri\) of the form \(s\cdot f\) with \(s\in S\) and an invertible morphism~\(f\) in~\(\Tri\).  It is easy to see that~\(W\) has both left and right calculi of fractions with localisation~\(\TriS\).

  It remains to show that the category theoretic localisations of~\(\Tri\) at~\(W\) and at the class of all \(S\)\nb-equivalences are equal.  That is, a functor defined on~\(\Tri\) maps all elements of~\(W\) to isomorphisms if and only if it maps all \(S\)\nb-equivalences to isomorphisms.  One direction is clear because~\(W\) consists of \(S\)\nb-equivalences.  Conversely, if~\(f\) is an \(S\)\nb-equivalence, then there are morphisms \(g\) and~\(h\) such that \(fg\) and~\(hf\) belong to~\(W\).  Hence a functor that maps~\(W\) to isomorphisms maps \(fg\) and~\(hf\) to isomorphisms.  Then it maps~\(f\) to an isomorphism as well.
\end{proof}

\begin{defn}
  \label{def:localise_homological}
  Let \(F\colon \Tri\to\Ab\) be a homological functor.  Its \emph{localisation at~\(S\)} is the functor
  \[
  S^{-1}F\colon \Tri\to\Ab,\qquad
  S^{-1}F(A) \defeq F(A)\otimes_\Zen \ZenS.
  \]
\end{defn}

Since~\(\ZenS\) is a flat \(\Zen\)\nb-module, the functor~\(S^{-1}F\) is again homological.  By definition, the localisation of the homological functor \(\Tri(A,\blank)\) for an object~\(A\) of~\(\Tri\) agrees with~\(\TriS(A,\blank)\).

\begin{prop}
  \label{pro:localisation_functor}
  The functor~\(S^{-1}F\) is the localisation of~\(F\) with respect to the thick subcategory of \(S\)\nb-finite objects.
\end{prop}

\begin{proof}
  The localisation of~\(F\) at the \(S\)\nb-finite objects maps an object~\(A\) of~\(\Tri\) to the colimit of \(F(B)_f\), where~\(f\) runs through the directed set of \(S\)\nb-equivalences \(f\colon A\to B\) (see also~\cite{Inassaridze-Kandelaki-Meyer:Localisation_colocalisation}).  As in the proof of Theorem~\ref{the:TriS_localisation_triangulated}, we may replace \(S\)\nb-equivalences by \(s\cdot\id_A\) for \(s\in S\), where we use the filtered category~\(\Cats\) introduced in Definition~\ref{def:Cats}.  The colimit over~\(\Cats\) agrees with \(S^{-1}F(A)\) by Lemma~\ref{lem:localisation_as_colimit}.
\end{proof}

\subsection{The colocalisation}
\label{sec:colocalisation}

Following~\cite{Inassaridze-Kandelaki-Meyer:Localisation_colocalisation}, we embed the functor \(\Tri\to\TriS\) into an exact sequence
\begin{multline}
  \label{eq:loc_coloc_sequence}
  \dotsb \to
  \Tri_1(A,B) \to
  \TriS_1(A,B) \to
  \Tri_1(A,B;\ZenS/\Zen) \\\to
  \Tri_0(A,B) \to
  \TriS_0(A,B) \to
  \Tri_0(A,B;\ZenS/\Zen) \to
  \Tri_{-1}(A,B) \to \dotsb.
\end{multline}
In the notation of~\cite{Inassaridze-Kandelaki-Meyer:Localisation_colocalisation}, \(\TriS(A,B) \defeq \Tri/\FinS(A,B)\) and \(\Tri_{n+1}(A,B;\ZenS/\Zen) \defeq (\Tri/\FinS)^\bot_n(A,B)\), where~\(\FinS\) denotes the class of \(S\)\nb-finite objects.  The degree shift is natural if we think of \(\Tri_*(A,B)\) and \(\TriS_*(A,B)\) as~\(\Tri\) with coefficients in \(\Zen\) and~\(\ZenS\).

More generally, if \(F\colon \Tri\to\Ab\) is a homological functor, then the map \(F\to S^{-1}F\) embeds in an exact sequence
\begin{multline}
  \label{eq:loc_coloc_sequence_fun}
  \dotsb \to
  F_1(A) \to
  S^{-1} F_1(A) \to
  F_1(A;\ZenS/\Zen) \\\to
  F_0(A) \to
  S^{-1} F_0(A) \to
  F_0(A;\ZenS/\Zen) \to
  F_{-1}(A) \to
  S^{-1} F_{-1}(A) \to \dotsb
\end{multline}
with \(F_{n+1}(A;\ZenS/\Zen) \defeq \Rightco_n F(A)\) in the notation of~\cite{Inassaridze-Kandelaki-Meyer:Localisation_colocalisation}.

\begin{prop}
  \label{pro:nat_trafo_loc_coloc}
  A natural transformation \(\Phi\colon F\Rightarrow G\) is invertible if and only if both its colocalisation \(\Phi(\blank;\ZenS/\Zen)\colon F(\blank;\ZenS/\Zen) \Rightarrow G(\blank;\ZenS/\Zen)\) and its localisation \(S^{-1}\Phi\colon S^{-1}F\Rightarrow S^{-1}G\) are invertible.
\end{prop}

\begin{proof}
  This is \cite{Inassaridze-Kandelaki-Meyer:Localisation_colocalisation}*{Corollary 4.4}.
\end{proof}

The last result hints at an application of colocalisation and localisation: they break up computations in~\(\Tri\) into two hopefully simpler computations.  We will return to this point in some examples later.

\begin{prop}
  \label{pro:torsion_Tor}
  For an object~\(A\) of~\(\Tri\) and a homological functor \(F\colon \Tri\to\Ab\), we have \(\Tor_n^\Zen(F_*(A);\ZenS/\Zen)=0\) for all \(n\ge2\), and there is a natural group extension
  \[
  \Tor_0^\Zen(F_0(A);\ZenS/\Zen) \mono
  F_0(A;\ZenS/\Zen) \epi
  \Tor_1^\Zen(F_{-1}(A);\ZenS/\Zen).
  \]
\end{prop}

\begin{proof}
  Since \(\Zen\) and~\(\ZenS\) are flat \(\Zen\)\nb-modules, the defining exact sequence
  \[
  \dotsb \to 0 \to \Zen \to \ZenS\to \ZenS/\Zen
  \]
  is a flat resolution of~\(\ZenS/\Zen\).  We use this resolution to compute \(\Tor_n^\Zen(\blank;\ZenS/\Zen)\).  This yields the vanishing for \(n\ge2\) and
  \begin{equation}
    \label{eq:Tor}
    \begin{aligned}
      \Tor_1^\Zen(M;\ZenS/\Zen) &= \ker (M \to S^{-1}M),\\
      \Tor_0^\Zen(M;\ZenS/\Zen) &= \coker (M \to S^{-1}M).
    \end{aligned}
  \end{equation}
  The isomorphisms in~\eqref{eq:Tor} and the exact sequence~\eqref{eq:loc_coloc_sequence_fun} finish the proof.
\end{proof}

Proposition~\ref{pro:torsion_Tor} justifies our notation \(F(A;\ZenS/\Zen)\): we expect exactly such an extension for a homology theory with coefficients \(\ZenS/\Zen\).  The same remark applies to \(\Tri(A,B;\ZenS/\Zen)\) because it is a special case of \(F(B;\ZenS/\Zen)\).

It follows directly from the definition of~\(S^{-1}M\) that
\begin{equation}
  \label{eq:Tor_as_colimit}
  \begin{aligned}
    \ker (M \to S^{-1}M) &=
    \bigcup_{s\in S} \ker (M\xrightarrow{s} M),\\
    \coker (M \to S^{-1}M) &=
    \varinjlim_{s\inOb\Cats} \coker (M \xrightarrow{s} M),
  \end{aligned}
\end{equation}
where~\(s\) also denotes the endomorphism \(m\mapsto sm\) on~\(M\).  These endomorphisms form an inductive system indexed by~\(\Cats\) because of the commuting diagrams
\[
\xymatrix{
  M \ar[r]^{t} & M\\
  M \ar[r]^{s} \ar@{=}[u] & M \ar[u]_{u}
}
\]
for arrows \(u\colon s\to t\) in~\(\Cats\).  This explains the colimit of \(\coker (M \xrightarrow{s} M)\) for \(s\in\Cats\).  Of course, \(\coker (s\colon M\to M) = M/sM\).

\subsection{Homological functors with coefficients}
\label{sec:finite_homology}

Can we describe \(F(\blank;\ZenS/\Zen)\) as a filtered colimit, in analogy to~\eqref{eq:Tor_as_colimit}?  More precisely, given \(s\in S\), is there a homological functor \(A\mapsto F(A;s)\) that fits into an extension
\begin{equation}
  \label{eq:homology_coeff_s}
  \coker \bigl(s\colon F_0(A)\to F_0(A)\bigr) \mono
  F_0(A;s) \epi \ker\bigl(s\colon F_{-1}(A)\to F_{-1}(A)\bigr),
\end{equation}
and is \(F(A;\ZenS/\Zen)\) a filtered colimit of \(F(A;s)\) in a natural way?  If~\(s\) is not a zero divisor in~\(\Zen\), then~\eqref{eq:homology_coeff_s} is the expected behaviour for~\(F\) with coefficients in the \(\Zen\)\nb-module \(\Zen/s\Zen\); if~\(s\) is a zero divisor, then we should view \(F(A;s)\) as a theory with coefficients in the chain complex \(\Zen \xrightarrow{s}\Zen\) concentrated in degrees \(1\) and~\(0\).

Let~\(A_{s}\) be the cone of the map \(s\cdot\id_A\colon A\to A\).  Then the long exact homology sequence for~\(F\) applied to the defining exact triangle
\begin{equation}
  \label{eq:natural_cone_triangle}
  A \xrightarrow{s} A \to A_{s} \to A[1]
\end{equation}
yields an extension as in~\eqref{eq:homology_coeff_s}.  This suggests to define~\(F_0(A;s)\) as~\(F(A_{s})\).  However, without further assumptions it is not clear whether the construction of~\(A_{s}\) is natural, that is, whether \(F_*(A_{s})\) is functorial in~\(A\).  To proceed, we therefore \emph{assume that~\eqref{eq:natural_cone_triangle} is the object-part of a functor from~\(\Tri\) to the category of exact triangles in~\(\Tri\), such that the resulting functor \(A\mapsto A_{s}\) is triangulated.}  Then
\[
F(A;s)\defeq F(A_{s})
\]
defines a homological functor on~\(\Tri\) that fits into natural exact sequences as in~\eqref{eq:homology_coeff_s}.

\begin{rem}
  \label{rem:cone_functor_automatic}
  If~\(\Tri\) is a \emph{tensor} triangulated category and \(\Zen\subseteq\Tri(\Unit,\Unit)\), then the above assumption is automatic because we get~\eqref{eq:natural_cone_triangle} by tensoring the triangle \(\Unit \xrightarrow{s} \Unit \to \Unit_{s} \to \Unit[1]\) by~\(A\).
\end{rem}

\begin{lem}
  \label{lem:finite_from_torsion}
  Let \(F'_*(A) \defeq F_*(A;\ZenS/\Zen)\).  Then the natural transformation \(F'_{*+1}(A) \to F_*(A)\) induces natural isomorphisms \(F'_{*+1}(A;s) \cong F_*(A;s)\) for \(s\in S\).
\end{lem}

\begin{proof}
  Since~\(A_{s}\) is the cone of the \(S\)\nb-equivalence \(s\colon A\to A\), it is \(S\)\nb-finite.  Hence \(F_*(A_{s};\ZenS)=0\).  By the localisation--colocalisation exact sequence, this implies an isomorphism \(F_{*+1}(A_{s};\ZenS/\Zen) \cong F_*(A_{s})\).  This is equivalent to the assertion.
\end{proof}

Lemma~\ref{lem:finite_from_torsion} and~\eqref{eq:homology_coeff_s} yield a short exact sequence
\begin{multline}
  \label{eq:homology_coeff_s_tor}
  \coker \bigl(s\colon F_{*+1}(A;\ZenS/\Zen)\to F_{*+1}(A;\ZenS/\Zen)\bigr) \mono F_*(A;s)
  \\\epi \ker\bigl(s\colon F_*(A;\ZenS/\Zen)\to F_*(A;\ZenS/\Zen)\bigr),
\end{multline}

Next we want to write \(F(A;\ZenS/\Zen)\) as a filtered colimit of the finite coefficient theories \(F(A;s)\).  Given \(s,t\in S\) and \(u\in\Cats(s,t)\), that is, \(t=u\cdot s\), we may find a dotted arrow~\(\varphi\) that makes the following diagram commute:
\begin{equation}
  \label{eq:naturality_cone_s}
  \begin{aligned}
    \xymatrix{
      A \ar[r]^s \ar@{=}[d]&
      A \ar[r]\ar[d]^u&
      A_{s} \ar[r] \ar@{.>}[d]^\varphi&
      A[1] \ar@{=}[d]\\
      A \ar[r]^t&
      A \ar[r]&
      A_{t} \ar[r]&
      A[1]
    }
  \end{aligned}
\end{equation}
We let \(\Cats(A;s,t)\) be the set of all such pairs \((u,\varphi)\).

\begin{lem}
  \label{lem:filtered_Cats_A}
  The category with object set~\(S\) and morphism spaces \(\Cats(A;s,t)\) is filtered.
\end{lem}

\begin{proof}
  Any two objects in \(\Cats(A;s,t)\) are dominated by another object because~\(\Cats\) is filtered and an arrow~\(\varphi\) as above exists for any~\(u\).  Let \((u_1,\varphi_1)\) and \((u_2,\varphi_2)\) be two parallel arrows in \(\Cats(A;s,t)\).  In the construction of the localisation--colocaliation exact sequence in~\cite{Inassaridze-Kandelaki-Meyer:Localisation_colocalisation}, it is shown that we may equalise \((u_1,\varphi_1)\) and~\((u_2,\varphi_2)\) by a morphism of exact triangles
  \[
  \xymatrix{
    A \ar[r]^t\ar@{=}[d]&
    A \ar[r]\ar[d]^g&
    A_{t} \ar[r]\ar[d]^{\psi}&
    A[1]\ar@{=}[d]\\
    A \ar[r]^{tg}&
    A \ar[r]&
    A_{tg} \ar[r]&
    A[1]
  }
  \]
  where~\(g\) is an \(S\)\nb-equivalence.  Hence there is~\(g'\) such that \(g'g=v\in\Zen\).  We may embed~\(g'\) in a morphism of exact triangles \((g',\psi')\).  Then the composite morphism \((v,\psi'\circ\psi)\) belongs to \(\Cats(A;t,tv)\) and equalises \((u_1,\varphi_1)\) and~\((u_2,\varphi_2)\).  Thus the category \(\Cats(A)\) is filtered.
\end{proof}

\begin{prop}
  \label{pro:torsion_from_finite}
  The colimit of the inductive system of extensions~\eqref{eq:homology_coeff_s} is the extension in Proposition~\textup{\ref{pro:torsion_Tor}},
  \[
  \Tor_0^\Zen(F_0(A);\ZenS/\Zen) \mono
  F_0(A;\ZenS/\Zen) \epi
  \Tor_1^\Zen(F_{-1}(A);\ZenS/\Zen).
  \]
\end{prop}

\begin{proof}
  Recall that the localisation--colocalisation exact sequence in~\cite{Inassaridze-Kandelaki-Meyer:Localisation_colocalisation} is constructed as a colimit over the filtered category of all triangles \(N\to A\to B\to N[1]\) where the map \(A\to B\) is an \(S\)\nb-equivalence.  We get the same colimit if we use the filtered category with object set~\(S\) and morphisms \(\Cats(A;s,t)\) because the arrows \(s\colon A\to A\) with \(s\in S\) are cofinal in the filtered category of \(S\)\nb-equivalences \(A\to B\).  This observation is already implicit in the proof of Lemma~\ref{lem:filtered_Cats_A}.  Therefore, the colimit of the long exact sequences from the above filtered subcategory yields the localisation--colocalisation long exact sequence.  Splitting this into extensions then yields the assertion of the proposition.
\end{proof}

\begin{rem}
  \label{rem:natural_cone_Cats}
  The filtered category \(\Cats(A)\) with morphism spaces \(\Cats(A;s,t)\) used above may be rather large because we allow all pairs \((u,\varphi)\) making~\eqref{eq:naturality_cone_s} commute.  In nice cases, we expect a canonical choice~\(\varphi_u\) for each~\(u\), coming from a natural cone construction in some category whose homotopy category is~\(\Tri\).  Even more, this natural choice should define a functor \(\Cats\to \Cats(A)\), \(u\mapsto (u,\varphi_u)\).  If these nice things happen, then we may replace the filtered category \(\Cats(A)\) above by the filtered subcategory~\(\Cats\) because colimits over \(\Cats\) and~\(\Cats(A)\) agree.
\end{rem}

\begin{lem}
  \label{lem:square_annihilates_finite_coef}
  Multiplication by~\(s^2\) vanishes on \(F_*(A;s)\).
\end{lem}

\begin{proof}
  It is clear that multiplication by~\(s\) vanishes on the cokernel and kernel of multiplication by~\(s\) on \(F_*(A)\).  Hence the assertion follows from the exact sequence~\eqref{eq:homology_coeff_s}.
\end{proof}

\begin{rem}
  \label{rem:order_finite_coef}
  Sometimes the extension~\eqref{eq:homology_coeff_s} splits unnaturally (see, for instance, \cite{Schochet:Top4}).  In those cases, already multiplication by~\(s\) vanishes on \(F_*(A;s)\).
\end{rem}

\begin{lem}
  \label{lem:exact_sequence_coefficient}
  Let \(s,t\in S\).  Then there is a long exact sequence
  \begin{multline*}
    \dotsb \to F_0(A;s) \to F_0(A;st) \to F_0(A;t)
    \\\to F_{-1}(A;s) \to F_{-1}(A;st) \to F_{-1}(A;t) \to \dotsb,
  \end{multline*}
  which is natural in~\(F\) with respect to natural transformations.
\end{lem}

\begin{proof}
  The octahedral axiom for the maps \(s,t,st\colon A\to A\) yields a commuting diagram with exact rows and columns
  \[
  \xymatrix{
    A\ar[r]^s\ar@{=}[d]&
    A\ar[d]^t\ar[r]&
    A_{s}\ar[r]\ar[d]&
    A[1]\ar@{=}[d]\\
    A\ar[r]^{st} \ar[d]&
    A\ar[r] \ar[d]&
    A_{st}\ar[d] \ar[r]&
    A[1]\ar[d]\\
    0\ar[d]\ar[r]&
    A_{t}\ar@{=}[r]\ar[d]&
    A_{t}\ar[r]\ar[d]&
    0 \ar[d]\\
    A[1]\ar[r]^{s[1]}&
    A[1]\ar[r]&
    A_{s}[1]\ar[r]&
    A[2].
  }
  \]
  Here we use the uniqueness of cones to identify the objects \(A_{s}\), \(A_{t}\), and~\(A_{st}\) in the diagram.  Apply the functor~\(F\) to the third row to get the asserted long exact sequence for homology with coefficients.
\end{proof}

\begin{cor}
  \label{cor:iso_colocalisation}
  Let \(\Phi\colon F\Rightarrow G\) be a natural transformation between two homological functors on~\(\Tri\).  Let \(S_0\subseteq S\) be a generating subset, that is, any element of~\(S\) is a product of elements in~\(S_0\).

  Then~\(\Phi\) induces an invertible map \(F(A;\ZenS/\Zen)\to G(A;\ZenS/\Zen)\) for all~\(A\) if and only if~\(\Phi\) induces invertible maps \(F(A;s)\to G(A;s)\) for all \(s\in S_0\) and all objects~\(A\) of~\(\Tri\).
\end{cor}

\begin{proof}
  Assume first that the maps \(\Phi_{A,s}\colon F(A;s)\to G(A;s)\) are invertible for all \(s\in S_0\).  Let~\(S_1\) be the set of all \(s\in S\) for which~\(\Phi_{A,s}\) is invertible.  Lemma~\ref{lem:exact_sequence_coefficient} and the Five Lemma show that~\(S_1\) is closed under multiplication.  Hence~\(\Phi_{A,s}\) is invertible for all \(s\in S\).  Then Proposition~\ref{pro:torsion_from_finite} shows that~\(\Phi\) induces an invertible map \(F(A;\ZenS/\Zen)\to G(A;\ZenS/\Zen)\).

  For the converse implication, Lemma~\ref{lem:finite_from_torsion} allows us to replace~\(\Phi\) by the induced natural transformation~\(\Phi'\) between the functors
  \[
  F'_*(A) \defeq F_{*+1}(A;\ZenS/\Zen),\qquad
  G'_*(A) \defeq G_{*+1}(A;\ZenS/\Zen)
  \]
  because \(\Phi\) and~\(\Phi'\) induce equivalent maps between finite coefficient theories.  If~\(\Phi'\) is an isomorphism for all~\(A\), it is one for~\(A_{s}\), so that~\(\Phi_{A,s}\) is invertible.
\end{proof}

\section{Application to Kasparov theory}
\label{sec:applications}

Now we apply the general theory developed above to equivariant Kasparov theory, viewed as a triangulated category (see~\cite{Meyer-Nest:BC}).  We only consider central localisations where~\(\Zen\) is the ring~\(\Z\) of integers, as in Example~\ref{exa:integers}.  Finer information may be obtained by considering the larger ring \(\Rep(G)\) instead (see Example~\ref{exa:KKG_RepG}), but we leave this to future investigation.

Let us first consider the \emph{rational \(\KK^G\)-theory}.  Here \(S=\Z\setminus\{0\}\) and \(S^{-1}\Z=\Q\).  Following the definitions above, we define
\begin{equation}
  \label{eq:KK_rational}
  \KK^G_n(A,B;\Q)\defeq \KK^G_n(A,B)\otimes \Q,
\end{equation}
where \(A\) and~\(B\) are \(G\)\nb-\(\Cst\)-algebras.

This differs from the definition in \cite{Blackadar:K-theory}*{Exercise 23.15.6}, where rational \(\KK\)-theory for complex \(\Cst\)\nb-algebras is defined to be \(\KK_n(A,B \otimes D_\Q)\) for a \(\Cst\)\nb-algebra~\(D_\Q\) in the bootstrap class with \(\K_0(D_\Q) = \Q\) and \(\K_1(D_\Q) = 0\).  The example
\[
\KK_0(D_\Q,\C\otimes D_\Q) \cong\Q
\qquad\text{and}\qquad
\KK_0(D_\Q;\C)\otimes \Q = 0
\]
shows that the two theories are different.  In \(\KK(\blank,\blank;\Q)\), the \(\Cst\)\nb-algebras \(\C\) and~\(D_\Q\) are not isomorphic because there is no morphism from~\(D_\Q\) to~\(\C\).  The map \(\C\to D_\Q\) corresponding to the embedding \(\Z\to\Q\) is not an \(S\)\nb-equivalence.  Its cone is the suspension of a \(\Cst\)\nb-algebra~\(D_{\Q/\Z}\) with \(\K_0(D_{\Q/\Z}) = \Q/\Z\) and \(\K_1(D_{\Q/\Z}) = 0\).  This \(\Cst\)\nb-algebra is not \(S\)\nb-finite, so that \(\KK_0(D_{\Q/\Z},D_{\Q/\Z};\Q) \neq0\).  But \(D_{\Q/\Z} \otimes D_\Q\cong0\) in \(\KKcat\) because this tensor product belongs to the bootstrap category and its \(\K\)\nb-theory \(\K_*(D_{\Q/\Z} \otimes D_\Q) \cong \Q/\Z\otimes \Q\) vanishes.

The definition of \(\KK(A,B;\Q)\) above yields again a triangulated category.  This is crucial to apply methods from stable homotopy theory and homological algebra.  Does the definition in~\cite{Blackadar:K-theory} also yield a triangulated category?  This seems unclear if we consider all of \(\KKcat\); but we get a positive answer in the bootstrap class.  For~\(A\) in the bootstrap class, the Universal Coefficient Theorem yields
\[
\KK_0(A, B\otimes D_\Q) \cong \Hom(\K_*(A), \K_*(B)\otimes\Q) \cong
\Hom_\Q(\K_*(A)\otimes\Q, \K_*(B)\otimes\Q)
\]
because Abelian groups of the form \(\K_*(B)\otimes\Q\) are injective.  Hence the bootstrap class with these morphisms is equivalent to the category of countable \(\Q\)\nb-vector spaces.  This category is triangulated and Abelian at the same time.  And we may also view it as the localisation of~\(\KKcat\) at the class of \(\Cst\)\nb-algebras with vanishing rational \(\K\)\nb-theory \(\K_*(\blank)\otimes\Q\).  But this observation depends on an explicit computation of the category.  Once we go beyond the bootstrap category or consider equivariant situations, it is no longer clear whether the definition in~\cite{Blackadar:K-theory} yields a triangulated category.

More generally, if~\(S\) is any multiplicatively closed subset of~\(\Z\), then we may replace~\(\Q\) by~\(S^{-1}\Z\) and define
\[
\KK^G_*(A,B;S^{-1}\Z) = S^{-1} \KK^G_*(A,B) \defeq \KK^G_*(A,B) \otimes S^{-1}\Z.
\]
By our general theory, these groups form the morphism spaces of an \(S^{-1}\Z\)-linear triangulated category.  It is the localisation of \(\KKcat^G\) at the class of \(S\)\nb-finite \(G\)\nb-\(\Cst\)-algebras.  Here~\(A\) is \(S\)\nb-finite if and only if there is \(s\in S\) with \(s\cdot\id_A=0\).

The colocalisation also produces an \emph{\(S\)\nb-torsion theory} \(\KK^G_*(A,B;S^{-1}\Z/\Z)\) that fits into a natural long exact sequence
\begin{multline}
  \label{eq:torsion_rational_KK}
  \dotsb \to \KK^G_{n+1}(A,B)
  \to \KK^G_{n+1}(A,B;S^{-1}\Z)
  \to \KK^G_{n+1}(A,B;S^{-1}\Z/\Z)
  \\\to \KK^G_{n}(A,B)
  \to \KK^G_{n}(A,B;S^{-1}\Z)
  \to \KK^G_{n}(A,B;S^{-1}\Z/\Z)
  \to \dotsb.
\end{multline}
In particular, this includes a \emph{torsion theory} \(\KK^G_*(A,B;\Q/\Z)\).

The \(S\)\nb-rational and \(S\)\nb-torsion theories inherit basic properties like homotopy invariance, \(\Cst\)\nb-stability, excision and Bott periodicity from \(\KK^G\).  All this is contained in the statement that they are bifunctors on \(\KKcat^G\) that are homological in the first and cohomological in the second variable.  Furthermore, the maps in~\eqref{eq:torsion_rational_KK} are natural transformations.  Since the \(S\)\nb-rational theory is again a triangulated category, we get an associative product
\[
\KK^G_n(A,B;S^{-1}\Z)\otimes_{S^{-1}\Z} \KK^G_m(B,C;S^{-1}\Z) \to\KK^G_{n+m}(A,C;S^{-1}\Z).
\]
It is the product of the natural isomorphism
\begin{multline*}
  (\KK^G_n(A,B)\otimes S^{-1}\Z)\otimes_{S^{-1}\Z} (\KK^G_n(B,C)\otimes S^{-1}\Z)\\
  \cong \bigl(\KK^G_n(A,B)\otimes \KK^G_n(B,C)\bigr)\otimes S^{-1}\Z
\end{multline*}
and the homomorphism
\[
\KK^G_n(A,B)\otimes \KK^G_n(B,C)\otimes S^{-1}\Z \to
\KK^G_n(A,C)\otimes S^{-1}\Z
\]
induced by the Kasparov product.

As for the rational theory, the torsion \(\KK\)-theory \(\KK(A,B;\Q/\Z)\) is, in general, not isomorphic to \(\KK_*(A,B\otimes D_{\Q/\Z})\).

As in Section~\ref{sec:finite_homology}, we now realise the \(S\)\nb-torsion theory as a filtered colimit of finite coefficient Kasparov theories for \(s\in S\).  These functors have already been studied by Claude Schochet~\cite{Schochet:Top4}.  This is a nice case in the sense of Remark~\ref{rem:natural_cone_Cats}, that is, we get canonical natural transformations \(A_s \Rightarrow A_t\) for \(s,t\in S\) with \(s\mid t\).

Let~\(\Sphere^1\) be the unit circle in~\(\C\) with base point \(*=1\in\Sphere^1\).  The map
\[
\tilde{q}\colon \Sphere^1\to \Sphere^1,\qquad x\mapsto x^q,
\]
for \(q\in\N_{\ge2}\) is called \emph{standard \(q\)th power map}.  It preserves the base point.  Let \(\Cont_0(\Sphere^1) \cong \Cont_0(\R)\) be the \(\Cst\)\nb-algebra of continuous functions on~\(\Sphere^1\) vanishing at~\(*\).  The map~\(\tilde{q}\) induces a \(*\)\nb-homomorphism
\[
\hat{q}\colon \Cont_0(\Sphere^1)\to\Cont_0(\Sphere^1),\qquad
f(s)\mapsto f(s^q).
\]
Let~\(C_q\) be the mapping cone \(\Cst\)\nb-algebra of~\(\hat{q}\):
\[
C_q \defeq \{(x,f)\in \Cont_0(\Sphere^1)\oplus
\Cont_0(\Sphere^1)\otimes \Cont_0[0,1) \mid \hat{q}(x)=f(0)\}.
\]
Up to a desuspension, we may define the finite coefficient theories by
\[
A_q \defeq A \otimes C_q[-2],\qquad
F_*(A;q) \defeq F_{*-2}(A\otimes C_q).
\]
In particular, we define
\[
\KK^G_n(A,B;\Z/q) \defeq \KK^G_{n-2}(A,B\otimes C_q),
\]

In the complex case, the Puppe exact sequence and Bott periodicity imply \(\K_0(C_q) \cong\Z/q\) and \(\K_1(C_q) \cong0\).  Furthermore, \(C_q\) belongs to the bootstrap class and is the unique object of the bootstrap class with this \(\K\)\nb-theory.

By definition, \(C_q\) fits into a pull-back diagram
\[
\xymatrix{
  C_q\ar[r]\ar[d]^{q_0} & \Cont_0(\Sphere^1)\otimes \Cont_0[0,1)\ar[d]^{e_0}\\
  \Cont_0(\Sphere^1)\ar[r]^{\hat{q}}&\Cont_0(\Sphere^1).
}
\]
If~\(q\) divides~\(p\), then we get a natural commuting diagram
\[
\xymatrix{
  C_q\ar[r]\ar[d]_{\Theta_{p,q}}&\Cont_0(\Sphere^1)\ar[r]^{\hat{q}}\ar @{=}[d]&
  \Cont_0(\Sphere^1)\ar[d]^{\widehat{p/q}}\\
  C_p\ar[r]&\Cont_0(\Sphere^1)\ar[r]^{\hat{p}}&\Cont_0(\Sphere^1)
}
\]
We get \(\Theta_{p,r} \Theta_{r,q}=\Theta _{p,q}\) if \(q \mid r \mid p\) by \cite{Schochet:Top4}*{Proposition 2.1(3)}.  This provides a functor \(\Cats\to\KK^G\), \(s\mapsto A_s\), \(p/q\mapsto \Theta_{p,q}\).  We get a natural isomorphism
\[
\KK^G_n(A,B;S^{-1}\Z/\Z) \cong \varinjlim_{\Cats} \KK^G_n(A,B;\Z/s),
\]
that is, we may replace the filtered category~\(\Cats(A)\) that is used in Section~\ref{sec:finite_homology} by the much smaller category~\(\Cats\).  Furthermore, it can be checked that the maps \(F_*(A;s) \to F_*(A;st)\) in Lemma~\ref{lem:exact_sequence_coefficient} may be chosen to be the maps induced by~\(\Theta_{s,st}\).

The finite coefficient theory is related to \(\KK^G\) by a natural exact sequence
\begin{multline*}
  \dotsb
  \to \KK^G_n(A,B)
  \xrightarrow{q} \KK^G_n(A,B)
  \to \KK^G_n(A,B;\Z/q)
  \\\to \KK^G_{n-1}(A,B)
  \xrightarrow{q} \KK^G_{n-1}(A,B)
  \to \KK^G_{n-1}(A,B;\Z/q)
  \to \dotsb
\end{multline*}
for \(q\in S\) by~\eqref{eq:homology_coeff_s}; here~\(q\) denotes the map of multiplication by~\(q\).  Similarly, \eqref{eq:homology_coeff_s_tor} yields a natural exact sequence
\begin{multline*}
  \dotsb \to \KK^G_n(A,B;\Z/q)
  \to \KK^G_n(A,B;S^{-1}\Z/\Z)
  \xrightarrow{q} \KK^G_n(A,B;S^{-1}\Z/\Z)\to
  \\ \KK^G_{n-1}(A,B;\Z/q)
  \to \KK^G_{n-1}(A,B;S^{-1}\Z/\Z)
  \xrightarrow{q} \KK^G_{n-1}(A,B;S^{-1}\Z/\Z)
  \to \dotsb.
\end{multline*}
Lemma~\ref{lem:exact_sequence_coefficient} provides natural exact sequences
\begin{multline*}
  \dotsb \to \KK^G_n(A,B;\Z/p)
  \to \KK^G_n(A,B;\Z/pq)
  \to \KK^G_n(A,B;\Z/q)
  \\\to \KK^G_{n-1}(A,B;\Z/p)
  \to \KK^G_{n-1}(A,B;\Z/pq)
  \to \KK^G_{n-1}(A,B;\Z/q)
  \to \dotsb
\end{multline*}
for all \(p,q\in\Z\), see also~\cite{Schochet:Top4}*{Proposition 2.6.(1)}.  The maps in this sequence come from a natural exact triangle \(C_p \to C_{pq} \to C_q \to C_p[1]\)

The \emph{exponent} of a group~\(G\) is the smallest \(a\in\N_{\ge1}\) with \(a\cdot g=0\) for all \(g\in G\).  At first sight, it may seem plausible that the exponent of \(\KK^G_*(A,B;\Z/q)\) should divide~\(q\).  There is, however, an obstruction to this for real \(\Cst\)\nb-algebras.  The issue is discussed in \cite{Schochet:Top4}*{Proposition 2.4}.  Complex \(\KK\)\nb-theory is good in the sense of~\cite{Schochet:Top4} because the \(\KK\)-class of the Hopf map \(\Sphere^3\to\Sphere^2\) belongs to the group \(\KK_0\bigl(\Cont_0(\R^2),\Cont_0(\R^3)\bigr)=0\).  As a consequence, the exact sequence~\eqref{eq:homology_coeff_s} for complex \(\KK\)\nb-theory splits unnaturally and, in particular, multiplication by~\(q\) vanishes on \(\KK^G_*(A,B;\Z/q)\) for all \(q\in\Z\).  This also follows easily from the isomorphism \(\KK_0(C_q,C_q) \cong \Z/q\), which follows from the Universal Coefficient Theorem.

Now consider real \(\Cst\)\nb-algebras, denote their \(\KK\)-theory by \(\KKO\).  The extension in~\eqref{eq:homology_coeff_s} for \(\KKO^G_*(A,B)\) is equivalent to an extension
\begin{equation}
  \label{eq:KKO_with_coef_UCT}
  \KKO^G_*(A,B) \otimes \Z/q \mono \KKO^G_*(A,B;\Z/q) \epi \Tor(\KKO^G_{*-1}(A,B), \Z/q)  
\end{equation}
because
\[
H \otimes \Z/q \cong \coker (q\colon H\to H),\qquad
\Tor(H,\Z/q) \cong \ker (q\colon H\to H).
\]
The above extension splits unnaturally for odd~\(q\) by \cite{Schochet:Top4}*{Proposition 2.4}, so that \(\KKO^G_*(A,B;\Z/q)\) has exponent~\(q\).  But since the class of the Hopf map is the generator of \(\KKO_0\bigl(\Cont_0(\R^2),\Cont_0(\R^3)\bigr) \cong \Z/2\), the homology theory \(\KKO^G_*(A,\blank)\) is not good.  Hence the above extension need not split for even~\(q\).  Instead, general arguments from homotopy theory show that the exponent of \(\KKO^G_*(A,B;\Z/q)\) divides~\(2q\) for all even~\(q\), and it divides~\(q\) if \(4\mid q\).  We only explain why~\(2q\) annihilates \(\KKO^G_*(A,B;\Z/q)\) for even~\(q\); the stronger assertion for \(4\mid q\) is more difficult.

We assume that~\(q\) is even.  Since \(\KKO_{0}(C_{q},C_{q})=\KKO_{0}(C_q,\R;\Z/q)\), Equation~\eqref{eq:KKO_with_coef_UCT} yields exact sequences
\[
0 \to \KKO_{0}(C_q,\R)\otimes \Z/q\to \KKO_{0}(C_q,C_q)\to
\Tor(\KKO_{-1}(C_q,\R),\Z/q)\to 0
\]
for all \(q>1\).  Since \(\KKO_1(\R,\R)=\Z/2\), \(\KKO_0(\R,\R)=\Z\) and \(\KKO_{-1}(\R,\R)=0\), we get a long exact sequence
\[
\dotsb\to \Z/2\xrightarrow{q} \Z/2\to \KKO_0(C_q,\R)\to
\Z\xrightarrow{q} \Z\to \KKO_{-1}(C_q,\R)\to 0\to\dotsb
\]
Since \(q\colon \Z\to \Z\) is a monomorphism, we have exact sequences
\[
0\to \Z\xrightarrow{q} \Z\to \KKO_{-1}(C_q,\R)\to 0
\]
and
\[
\Z/2\xrightarrow{q} \Z/2\to \KKO_0(C_q,\R) \to 0
\]
Thus \(\KKO_{-1}(C_q,\R)=\Z/q\) and \(\KKO_0(C_q,\R)=\Z_2\) because~\(q\) is even.  Plugging this into the UCT exact sequence yields
\[
0\to \Z/2\otimes \Z/q\to \KKO_{0}(C_q,C_q)\to \Tor(\Z/q,\Z/q)\to 0.
\]
We have \(\Tor(\Z/q,\Z/q)=\Z/q\) and \(\Z/2\otimes \Z/q\cong \Z/2\) because~\(q\) is even.  Therefore, \(\KKO_{0}(C_q,C_q)\) is annihilated by~\(2q\).

Since \(\KKO^G(A,B; \Z/q)\) is a module over \(\KKO_0(C_q,C_q)\), it is annihilated by~\(q\) for odd~\(q\) and by~\(2q\) for odd~\(q\).

\subsection{Look at the Baum--Connes Conjecture}
\label{sec:BC}

The localisation--colocalisation exact sequence allows to reduce computations in \(\KKcat^G\) to computations in the \(S\)\nb-rational and \(S\)\nb-torsion variants of \(\KK^G\).  The computation for \(S\)\nb-torsion invariants, in turn, reduces to computations for \(\KK^G\) with finite coefficients.  Here we explain what this means for the Baum--Connes conjecture.  However, we know of no concrete application where this reduction of the problem would help to settle this conjecture.

The formulation of the Baum--Connes Conjecture involves the groups \(\K^\topo_n(G,A)\), called the \emph{topological \(\K\)\nb-theory of~\(G\) with coefficients in~\(A\)}, and a natural transformation
\[
\mu_A \colon  \K^\topo_*(G,A)\to \K_*(G \ltimes_\red A),
\]
called the \emph{Baum--Connes assembly map}.  The \emph{Baum--Connes Conjecture} for~\(G\) with coefficients in~\(A\) asserts that~\(\mu_A\) is an isomorphism.  Note that \(\K^\topo_*(G,\blank)\) and \(\K_*(G\ltimes_\red\blank)\) are homology theories of the kind studied in Section~\ref{sec:finite_homology} (see also~\cite{Meyer-Nest:BC}).  Therefore, the constructions above yield \(S\)\nb-rational, finite, and \(S\)\nb-torsion versions of the Baum--Connes assembly map:
\begin{alignat*}{2}
  \mu_A^{S^{-1}\Z} &\colon&
  \K^\topo_*(G,A; S^{-1}\Z) &\to
  \K_*(G \ltimes_\red A; S^{-1}\Z),\\
  \mu_A^{\Z/q} &\colon&
  \K^\topo_*(G,A;\Z/q) &\to \K_*(G \ltimes_\red A;\Z/q),\\
  \mu_A^{S^{-1}\Z/\Z} &\colon&
  \K^\topo_*(G,A;S^{-1}\Z/\Z) &\to \K_*(G \ltimes_\red A;S^{-1}\Z/\Z).
\end{alignat*}

\begin{thm}
  \label{bcceq}
  The following conjectures are equivalent:
  \begin{enumerate}[label=\textup{(\arabic{*})}]
  \item the Baum--Connes Conjecture with coefficients;
  \item the \(S\)\nb-rational and \(S\)\nb-torsion Baum--Connes Conjectures with coefficients;
  \item the \(S\)\nb-rational and \(q\)\nb-finite Baum--Connes Conjectures with coefficients for all primes \(q\in S\).
  \end{enumerate}
\end{thm}

\begin{proof}
  This follows immediately from Proposition~\ref{pro:nat_trafo_loc_coloc} and Corollary~\ref{cor:iso_colocalisation}.
\end{proof}

\subsection{Real versus complex Kasparov theory}
\label{sec:real_complex_KK}

To illustrate the usefulness of localisation, we reformulate some well-known results about the relationship between real and complex Kasparov theory and \(\K\)\nb-theory.  Roughly speaking, these two theories become equivalent when we localise at~\(2\), that is, work with \(\Z[\nicefrac12]\)-coefficients.  The results in this section are due to Max Karoubi~\cite{Karoubi:Descent} and Thomas Schick~\cite{Schick:Real_complex}.

In~\cite{Schick:Real_complex}, Thomas Schick related the \(\KK\)-theories of two real \(\Cst\)-algebras \(A\) and~\(B\) and their complexifications \(A_\C\) and~\(B_\C\) by an exact sequence
\begin{multline}
  \label{eq:real_complex_sequence}
  \dotsb\to \KKO^\Gamma_{n-1} (A,B)
  \xrightarrow{\chi} \KKO^\Gamma_n(A,B)
  \xrightarrow{c} \KK^\Gamma_n(A_\C,B_\C)
  \\\xrightarrow{\delta} \KKO^\Gamma_{n-2}(A,B)
  \xrightarrow{\chi} \KKO^\Gamma_{n-1}(A,B)
  \xrightarrow{c} \KK^\Gamma_{n-1}(A_\C,B_\C)
  \to\dotsb,
\end{multline}
extending previous results for real and complex \(\K\)\nb-theory, \(\KO\) and~\(\K\).  In~\cite{Schick:Real_complex}, \(\Gamma\) is assumed to be a discrete group, but the same arguments work if~\(\Gamma\) is replaced by a locally compact group or even groupoid; \(A\) and~\(B\) are separable real \(\Gamma\)\nb-\(\Cst\)-algebras; \(\chi\) is given by Kasparov product with the generator of \(\KKO^\Gamma_{1}(\R,\R)=\Z/2\); \(c\) is the complexification functor; and~\(\delta\) is the composition of the inverse of the complex Bott periodicity isomorphism with ``forgetting the complex structure.''

The proof of \cite{Schick:Real_complex}*{Corollary 2.4} shows that, after inverting~\(2\), the long exact sequence above becomes a \emph{naturally split} short exact sequence
\[
\KKO^\Gamma_n(A,B)\otimes\Z[\nicefrac12]
\overset{c}{\mono} \KK^\Gamma _n (A_\C,B_\C)\otimes\Z[\nicefrac12]
\overset{\delta}{\epi} \KKO^\Gamma_{n-2}(A,B)\otimes\Z[\nicefrac12].
\]
In our notation, this yields a natural isomorphism
\[
\KK^\Gamma _n (A_\C,B_\C; \Z[\nicefrac12])
\cong \KKO^\Gamma_n(A,B;\Z[\nicefrac12]) \oplus \KKO^\Gamma_{n-2}(A,B;\Z[\nicefrac12]).
\]
Besides the exactness of~\eqref{eq:real_complex_sequence}, the proof uses two observations.  First, \(2\chi=0\), so that~\eqref{eq:real_complex_sequence} splits into short exact sequences after tensoring with \(\Z[\nicefrac12]\).  Secondly, the map
\[
\KKO^\Gamma_{n-2}(A,B) \xrightarrow{c}
\KK^\Gamma_{n-2}(A_\C,B_\C) \cong
\KK^\Gamma_{n}(A_\C,B_\C) \xrightarrow{\delta}
\KKO^\Gamma_{n-2}(A,B),
\]
where the middle isomorphism is Bott periodicity, is multiplication by~\(2\) (\cite{Schick:Real_complex}*{Lemma 3.9}).  Hence~\(c\) provides a natural section for the resulting extensions, up to inverting~\(2\).  More generally, the same argument yields:

\begin{thm}
  Let~\(G\) be a second countable locally compact group, let \(A\) and~\(B\) be separable real \(G\)\nb-\(\Cst\)-algebras.  There is a natural isomorphism
  \[
  \KK^\Gamma _n (A_\C,B_\C; H)
  \cong \KKO^\Gamma_n(A,B; H) \oplus \KKO^\Gamma_{n-2}(A,B; H)
  \]
  for the following coefficients:
  \begin{enumerate}[label=\textup{(\arabic{*})}]
  \item \(H=S^{-1}\Z\) with \(2\in S\) \textup(localisation\textup);
  \item \(H=\Z/s\Z\) with odd~\(s\) \textup(finite coefficients\textup);
  \item \(H=S^{-1}\Z/\Z\) if~\(S\) contains only odd numbers \textup(colocalisation\textup).
  \end{enumerate}
\end{thm}

\begin{proof}
  Tensoring~\eqref{eq:real_complex_sequence} with \(S^{-1}\Z\), we get an analogous long exact sequence for the \(S\)\nb-rational theories.  By assumption on~\(S\), \(2\)~is invertible on the \(S\)\nb-rational theory.  Hence exactly the same argument as for \(\Z[\nicefrac12]\) works.  Studying \(\KK\)-theories with coefficients in \(\Z/s\) amounts to replacing~\(B\) by \(B_{s} \defeq B\otimes \R_{s}\), an operation that commutes with complexification and the various other constructions needed for the exact sequence~\eqref{eq:real_complex_sequence}.  Hence there is an analogous exact sequence with coefficients~\(\Z/s\).  Since~\(s^2\) annihilates the theory with coefficients~\(\Z/s\) by Lemma~\ref{lem:square_annihilates_finite_coef}, the element~\(s^2+1\) acts as the identity  on~\(\Z/s\), and~\(2\)~is invertible in~\(\Z/s\) for odd~\(s\).  This proves the second case.  Finally, we write the \(S\)\nb-torsion theory as a filtered colimit of finite coefficient theories as in Proposition~\ref{pro:torsion_from_finite}.  Only coefficients~\(\Z/s\) with odd~\(s\) appear here by assumption.  Hence multiplication by~\(2\) is invertible on the \(S\)\nb-torsion theory if~\(S\) contains no even numbers.  This establishes the third case.
\end{proof}


\begin{bibdiv}
  \begin{biblist}
\bib{Atiyah-Macdonald:Commutative}{book}{
  author={Atiyah, Michael F.},
  author={Macdonald, Ian G.},
  title={Introduction to commutative algebra},
  publisher={Addison-Wesley Publishing Co., Reading, Mass.-London-Don Mills, Ont.},
  date={1969},
  pages={ix+128},
  review={\MRref {0242802}{39\,\#4129}},
}

\bib{Balmer:Spectra_spectra}{article}{
  author={Balmer, Paul},
  title={Spectra, spectra, spectra},
  status={preprint},
  date={2009},
  note={available at \url {http://www.math.ucla.edu/~balmer/research/Pubfile/SSS.pdf}},
}

\bib{Blackadar:K-theory}{book}{
  author={Blackadar, Bruce},
  title={\(K\)\nobreakdash -theory for operator algebras},
  series={Mathematical Sciences Research Institute Publications},
  volume={5},
  edition={2},
  publisher={Cambridge University Press},
  place={Cambridge},
  date={1998},
  pages={xx+300},
  isbn={0-521-63532-2},
  review={\MRref {1656031}{99g:46104}},
}

\bib{Browder:Algebraic_K_coefficients}{article}{
  author={Browder, William},
  title={Algebraic $K$\nobreakdash -theory with coefficients~$\mathbb Z/p$},
  conference={ title={Geometric applications of homotopy theory (Proc. Conf., Evanston, Ill., 1977), I}, },
  book={ series={Lecture Notes in Math.}, volume={657}, publisher={Springer}, place={Berlin}, },
  date={1978},
  pages={40--84},
  review={\MRref {513541}{80b:18011}},
}

\bib{Cuntz-Meyer-Rosenberg}{book}{
  author={Cuntz, Joachim},
  author={Meyer, Ralf},
  author={Rosenberg, Jonathan M.},
  title={Topological and bivariant \(K\)\nobreakdash -theory},
  series={Oberwolfach Seminars},
  volume={36},
  publisher={Birkh\"auser Verlag},
  place={Basel},
  date={2007},
  pages={xii+262},
  isbn={978-3-7643-8398-5},
  review={\MRref {2340673}{2008j:19001}},
}

\bib{dellAmbroglio:Tensor_triangular}{article}{
  author={dell'Ambroglio, Ivo},
  title={Tensor triangular geometry and $KK$-theory},
  status={preprint},
  date={2009},
  note={available at \url {http://www.math.ethz.ch/u/ambrogio/kkGarticle.pdf}},
}

\bib{Gabriel-Zisman:Fractions}{book}{
  author={Gabriel, Peter},
  author={Zisman, Michel},
  title={Calculus of fractions and homotopy theory},
  series={Ergebnisse der Mathematik und ihrer Grenzgebiete},
  volume={35},
  publisher={Springer-Verlag},
  place={New York},
  date={1967},
  pages={x+168},
  review={\MRref {0210125}{35\,\#1019}},
}

\bib{Inassaridze-Kandelaki-Meyer:Localisation_colocalisation}{article}{
  author={Inassaridze, Hvedri},
  author={Kandelaki, Tamaz},
  author={Meyer, Ralf},
  title={Localisation and colocalisation of triangulated categories at thick subcategories},
  note={\arxiv {0912.2088}},
  date={2009},
}

\bib{Karoubi:Descent}{article}{
  author={Karoubi, Max},
  title={A descent theorem in topological $K$\nobreakdash -theory},
  journal={$K$\nobreakdash -Theory},
  volume={24},
  date={2001},
  number={2},
  pages={109--114},
  issn={0920-3036},
  review={\MRref {1869624}{2002m:19005}},
  doi={10.1023/A:1012785711074},
}

\bib{Karoubi-Lambre:Classes_caracteristiques}{article}{
  author={Karoubi, Max},
  author={Lambre, Thierry},
  title={Quelques classes caract\'eristiques en th\'eorie des nombres},
  language={French, with English summary},
  journal={J. Reine Angew. Math.},
  volume={543},
  date={2002},
  pages={169--186},
  issn={0075-4102},
  review={\MRref {1887882}{2003f:11178}},
}

\bib{Kasparov:Operator_K}{article}{
  author={Kasparov, Gennadi G.},
  title={The operator \(K\)\nobreakdash -functor and extensions of \(C^*\)\nobreakdash -algebras},
  language={Russian},
  journal={Izv. Akad. Nauk SSSR Ser. Mat.},
  volume={44},
  date={1980},
  number={3},
  pages={571--636, 719},
  issn={0373-2436},
  translation={ language={English}, journal={Math. USSR-Izv.}, volume={16}, date={1981}, number={3}, pages={513--572 (1981)}, },
  review={\MRref {582160}{81m:58075}},
}

\bib{Meyer-Nest:BC}{article}{
  author={Meyer, Ralf},
  author={Nest, Ryszard},
  title={The Baum--Connes conjecture via localisation of categories},
  journal={Topology},
  volume={45},
  date={2006},
  number={2},
  pages={209--259},
  issn={0040-9383},
  review={\MRref {2193334}{2006k:19013}},
}

\bib{Schick:Real_complex}{article}{
  author={Schick, Thomas},
  title={Real versus complex $K$\nobreakdash -theory using Kasparov's bivariant $KK$-theory},
  journal={Algebr. Geom. Topol.},
  volume={4},
  date={2004},
  pages={333--346},
  issn={1472-2747},
  review={\MRref {2077669}{2005f:19007}},
}

\bib{Schochet:Top4}{article}{
  author={Schochet, Claude},
  title={Topological methods for $C^*$\nobreakdash -algebras. IV. Mod~$p$ homology},
  journal={Pacific J. Math.},
  volume={114},
  date={1984},
  number={2},
  pages={447--468},
  issn={0030-8730},
  review={\MRref {757511}{86g:46103}},
}

\bib{Segal:Representation_ring}{article}{
  author={Segal, Graeme},
  title={The representation ring of a compact Lie group},
  journal={Inst. Hautes \'Etudes Sci. Publ. Math.},
  number={34},
  date={1968},
  pages={113--128},
  issn={0073-8301},
  review={\MRref {0248277}{40\,\#1529}},
}
  \end{biblist}
\end{bibdiv}
\end{document}